\documentclass[12pt,a4paper]{article}
\usepackage{amsmath,amssymb,amstext,amsfonts,latexsym,amsthm}
\usepackage{dsfont,mathrsfs,txfonts}
\usepackage{graphicx,float,hyperref,xcolor}
\usepackage[affil-sl]{authblk}
\usepackage[english,activeacute]{babel}
\usepackage{inputenc}
\usepackage[shortlabels]{enumitem}% itemise noident ([leftmargin=*])
\usepackage{orcidlink}

\newtheorem{theorem}{Theorem}
\newtheorem{lemma}{Lemma}[section]
\newtheorem{proposition}{Proposition}[section]
\newtheorem{definition}{Definition}[section]
\newtheorem{corollary}{Corollary}[theorem]

\newtheorem{example}{Example}
\newcounter{Th-Alfa}

\newcommand{\CC}{\mathds{C}}

\newcommand{\NN}{\mathds{N}}
\newcommand{\PP}{\mathds{P}}

\newcommand{\RR}{\mathds{R}}
\newcommand{\ZZ}{\mathds{Z}}
\newcommand{\ZZp}{\mathds{Z}_+}
\newcommand{\dsty}{\displaystyle}

\newcommand{\supp}{\mathop{\rm supp}}

\newcommand{\inter}[1]{#1\strut^{\mathrm{o}}\!}%Interior
\newcommand{\ch}[1]{\mathbf{Co}\!\left(#1 \right)} %Convex hull
\def\bbuildrel#1_#2^#3{\mathrel{
 \mathop{\kern 0pt#1}\limits_{#2}^{#3}}}
\def\bbbuildrel#1_#2{\mathrel{
 \mathop{\kern 0pt#1}\limits_{#2}}}

\newcommand{\Nzero}[2]{\mathbf{N}_z\!\left(#1;#2\right)}%-------------------Cantidad de zeros
\newcommand{\Ncero}[2]{\mathbf{N}_{o}\!\left(#1;#2\right)}%-------------------Cantidad de zeros

\newcommand{\funD}[4]{\begin{cases} #1 , & \hbox{if } \; #2;\\%-------------------Funciones por partes 2
										#3, & \hbox{if } \; #4 \end{cases}}
\def\bbuildrel#1_#2^#3{\mathrel{
 \mathop{\kern 0pt#1}\limits_{#2}^{#3}}}
\def\bbbuildrel#1_#2{\mathrel{
 \mathop{\kern 0pt#1}\limits_{#2}}}

\newcommand{\RRp}{\mathds{R}_{+}}

\newcommand{\sob}{{\mathsf{s}}}%-----------------------------------------------------------S de sobolev
\newcommand{\Ip}[2]{\langle #1,#2\rangle}%---------------------------------------------Producto interno
\newcommand{\IpS}[2]{\Ip{#1}{#2}_{\sob}}%--------------------------------------Producto de Sobolev
\title{\bf Sequentially-ordered Sobolev inner product and Laguerre-Sobolev polynomials}

\author[1]{Abel D\'{\i}az-Gonz\'{a}lez\,\orcidlink{0000-0001-6226-1925}\thanks{abel.diaz.gonzalez@vanderbilt.edu}}
\author[2]{Juan Hern\'{a}ndez\,\orcidlink{0000-0002-8465-0646}\thanks{ jhernandez14@uasd.edu.do}\thanks{{The research of J.T.M.  was   partially supported by   Fondo Nacional de Innovaci\'{o}n  y Desarrollo Cient\'{\i}fico y Tecnol\'{o}gico (FONDOCYT),  Dominican Republic, under grant   2020-2021-
1D1-135.}}}
\author[3]{H\'{e}ctor Pijeira-Cabrera\,\orcidlink{0000-0001-5141-9395}\thanks{hpijeira@math.uc3m.es}}
\affil[1]{Vanderbilt University, Nashville, Tennessee 37240, USA.}
\affil[2]{Escuela de Matem\'{a}ticas, Facultad de Ciencias, Universidad Aut\'{o}noma de Santo Domingo, \linebreak Santo Domingo 10105, Dominican Republic.}
\affil[3]{Departamento de Matem\'{a}ticas, Universidad Carlos III de Madrid,  \linebreak Legan\'{e}s  28911 ,  Madrid,  Spain.}

\date{}

\begin{document}
%%%%%%%%%%%%%%%%%%%%%%%%%%%%%%%%%%%%%%%%%%%%%%%%%%%%%%%%%%%%%%%%%%%%%%%%%%%%%%%%%%%%%%%%%
\maketitle

%%%%%%%%%%%%%%%%%%%%%%%%%%%%%%%%%%%%%%%%%%%%%%%%
\begin{abstract}
We study the sequence of polynomials $\{S_n\}_{n\geq 0}$ that are orthogonal with respect to the general discrete Sobolev-type inner product 
$$
\langle f,g \rangle_{\mathsf{s}}=\!\int\! f(x) g(x)d\mu(x)+\sum_{j=1}^{N}\sum_{k=0}^{d_j}\lambda_{j,k} f^{(k)}(c_j)g^{(k)}(c_j),
$$
where $\mu$ is a finite Borel measure whose support $\supp{\mu}$ is an infinite set of the real line, $\lambda_{j,k}\geq 0$, and the mass points $c_i$, $i=1,\dots,N$ are real values outside the interior of the convex hull of $\supp{\mu}$ ($c_i\in\RR\setminus\inter{\ch{\supp{\mu}}}$).  Under some restriction of order in the discrete part of $\IpS{\cdot}{\cdot}$, we prove that $S_n$ has at least $n-d^*$ zeros on $\inter{\ch{\supp{\mu}}}$, being $d^*$ the number of terms in the discrete part of $\IpS{\cdot}{\cdot}$. Finally, we obtain the outer relative asymptotic for $\{S_n\}$ in the case that the measure $\mu$ is the classical Laguerre measure, and for each mass point, only one order derivative appears in the discrete part of $\IpS{\cdot}{\cdot}$.

\bigskip

 

\end{abstract}

%%%%%%%%%%%%%%%%%%%%%%%%%%%%%%%%%%%%%%%%%%%%%%%%%%%%%%%%%%%%%%%%%%%%%%%%%%%%%%%%%%%%%%%%%
%%%%%%%%%%%%%%%%%%%%%%%%%%%%%%%%%%%%%%%%%%
\section{Introduction}
%%%%%%%%%%%%%%%%%%%%%%%%%%%%%%%%%%%%%%%%%%

Let $\mu$ be a positive finite Borel measure with finite moments, whose support $\Delta \subset \RR$  contains infinitely many points. We will denote by  $\ch{A}$  the convex hull of a set $A$ and  by $\inter{A}$ its~interior.

Let  $\{P_n\}_{n\geq 0}$ be  the monic orthogonal polynomial sequence with respect to the inner product
\begin{align*}
 \Ip{f}{g}_{\mu}= & \int_{\Delta} f(x) g(x) d\mu(x).
\end{align*}

 {An}  inner product is called standard if the multiplication operator is symmetric with respect to the inner product. Obviously, $\langle x f,g\rangle_{\mu}=\langle  f,x g\rangle_{\mu}$, i.e.,~ $\langle \cdot,\cdot \rangle_{\mu}$ is standard.  Significant parts of the applications of orthogonal polynomials in mathematics and  particular sciences are based on the following three consequences of this~fact.
\begin{enumerate}
\item  The polynomial $P_{n}$ has exactly $n$ real simple zeros in $\inter{\ch{\Delta}}$.  Moreover,   there is a zero of  $P_{n-1}$ between any two consecutive zeros of  $P_{n}$.
  \item  The three-term recurrence relation
\begin{equation*}\label{3TRRPn}
   xP_{n}(x)=P_{n+1}(x)+\beta _{n}P_{n}(x)+\gamma^2 _{n}P_{n-1}(x); \quad  P_{0} (x)=1, \; P_{-1}(x)=0,
\end{equation*}
where $\gamma_n={\|P_{n}\|_{\mu}}/{\|P_{n-1}\|_{\mu}}$ for $n\geq 1$,  ${\beta}_n={\Ip{P_n}{xP_n}_{\mu} }/{\|P_{n}\|_{\mu}^2}$ and  $\|\cdot\|_{\mu}=\sqrt{\Ip{\cdot}{\cdot}_{\mu} }$ denotes the norm induced by $\langle \cdot,\cdot \rangle_{\mu}$.
 \item For the kernel polynomials
\begin{equation}\label{Kernel-n}
K_{n}(x,y)=\sum_{k=0}^{n}\frac{P_{k}(x)P_{k}(y)}{\left\|
P_{k}\right\|_{\mu} ^{2}},
\end{equation} we have the Christoffel--Darboux identities
\begin{align}\label{CD-Ident}
   K_{n}(x,y)=\funD{\frac{P_{n+1}(x)P_{n}(y)-P_{n+1}(y)P_{n}(x)}{\|P_{n}\|_{\mu}^{2}\, (x-y)}}{x\neq y}{\frac{P_{n+1}^{\prime}(x)P_{n}(x)-P_{n+1}(x)P_{n}^{\prime}(x)}{\|P_{n}\|_{\mu}^{2}}}{x=y.}
  \end{align}
  \end{enumerate}

  {These}  identities play a fundamental role in the treatment of Fourier expansions with respect to a system of orthogonal polynomials (see \cite{Osil99}, Section 2.2). For~a review  of the use of  \eqref{Kernel-n} and \eqref{CD-Ident} in the spectral theory of orthogonal polynomials, we refer the reader to~\cite{Simon08}.   In~addition, see  the usual references~\cite{Chi78, Freud71, Szego75},  for~a basic background on  these and other  properties of $\{P_n\}_{n\geq 0}$.

Let $(a,b)=\inter{\ch{\supp{\mu}}}$, $N,d_j \in \ZZ_+$, $\lambda_{j,k} \geq 0$,  for~$j=1, 2, \dots, N$, $k=0,1,\dots,d_j$, $\{c_1,c_2,\dots,c_N\}\subset \RR \!\setminus\!(a,b)$, where $c_i \neq c_j$ if $i \neq j$ and $I_+=\{(j,k): \lambda_{j,k}>0\}$.  We consider the following \emph{Sobolev-type (or discrete Sobolev) inner product}
\begin{align}\nonumber
	\IpS{f}{g}&=\!\int\! f(x) g(x)d\mu(x)+\sum_{j=1}^{N}\sum_{k=0}^{d_j}\lambda_{j,k} f^{(k)}(c_j)g^{(k)}(c_j)\\ \label{GeneralSIP}
	&=\!\int\!f(x) g(x)d\mu(x)+\sum_{(j,k)\in I_+}\!\!\!\!\lambda_{j,k} f^{(k)}(c_j)g^{(k)}(c_j),
\end{align}	
where $f^{(k)}$ denotes the $k$-th derivative of the function $f$. Without~loss of generality,  we also assume $\{(j,d_j)\}_{j=1}^N\subset I_+$ and $d_1\leq d_2\leq \cdots \leq d_N$.  For~$n \in \ZZp$,  we shall denote by $S_n$ the monic polynomial of the lowest degree satisfying
\begin{equation}\label{Sobolev-Orth}
\IpS{x^k}{S_n}= 0, \quad  \text{for } \; k=0,1,\dots,n-1.
\end{equation}

 {It}  is easy to see that for all  $n\geq 0$, there exists such a unique polynomial $S_n$ of degree $n$. This is deduced by solving a homogeneous linear system with $n$ equations and $n+1$ unknowns. Uniqueness follows from the minimality of the degree for the polynomial solution. We refer the reader to~\cite{MaXu15,And01}  for a review of  this type of non-standard~orthogonality.

Clearly,    \eqref{GeneralSIP} is not standard, i.e.,~$\IpS{xp}{q}\neq\IpS{p}{xq}$, for~some $p,q\in \PP$.
 It is well known that the properties of orthogonal polynomials with respect to standard inner products differ  from those of the Sobolev-type polynomials. In~particular, the~zeros of the Sobolev-type polynomials can be complex, or if~real, they can be located outside the convex hull of the support of the measure $\mu$, as~can be seen in the following example.
\begin{example}[Zeros outside the convex hull of the measures supports]  Let
 \begin{align*}
   \IpS{f}{g}= & \int_{0}^{\infty} f(x) g(x) e^{-x}dx+2f'(-1)g'(-1),
 \end{align*}
 then the corresponding second-degree monic Sobolev-type orthogonal polynomial is $S_2(z)=z^2-2$,  whose zeros are $z_{1,2}=\pm \sqrt{2}$. Note that $-\sqrt{2} \not \in [-1,\infty)$.
\end{example}

Let $\{Q_{n}\}_{n\geq 0}$ be the sequence of monic orthogonal polynomials with respect to the inner product
\begin{align*}%\label{Mod-InnerP}
  \Ip{f}{g}_{\mu_\rho} = & \int f(x)\;g(x)\;d\mu_{\rho}(x), \; \text{ where } \; \rho(x)= \prod_{c_j\leq a}^{} \left( x-c_j\right)^{d_j+1}\!\!\prod_{c_j\geq b}^{} \left( c_j-x\right)^{d_j+1} \\ \nonumber
   &\text{ and } \; d\mu_{\rho}(x)= \rho(x)d\mu(x).
\end{align*}

 {Note}  that $\rho$ is a polynomial of degree $d=\sum_{j=1}^N(d_j+1)$, which is positive on $(a,b)$.  If~$n >d$, from~ \eqref{Sobolev-Orth}, $\{S_n\}$  satisfies the following quasi-orthogonality relations with respect to~$\mu_{\rho}$
\begin{equation*}%\label{quasi-orthogonal}
\Ip{S_n}{f}_{\mu_\rho}  = \Ip{S_n}{\rho f}_{\mu} = \int S_n(x) f(x) \rho(x) d\mu(x) = \IpS{S_n}{\rho f}=0 ,
\end{equation*}
for $f\in \PP_{n-d-1}$, where $\PP_n$ is the linear space of polynomials with real coefficients and the degree  at most  $n\in \ZZp$. Hence, \emph{the polynomial $S_n$ is quasi-orthogonal of order $d$ with respect to $\mu_{\rho}$} and by this argument, we obtain that  $S_n$ has at least $(n-d)$ changes of sign in $(a,b)$.

The results obtained for measures $\mu$ with bounded support (see  \cite{LopMarVan95}, (1.10)) suggest that the number of zeros located in the interior of the support of the measure is closely related to $d^*=|I_+|$, the~number of terms in the discrete part of $\IpS{\cdot}{\cdot}$ (i.e., $\lambda_{j,k}>0$), instead of this greater quantity $d$.

Our first result, Theorem \ref{Th_ZerosSimp}, goes in this direction   for the case when  the inner product is sequentially ordered. This kind of inner product  is introduced in Section~\ref{Sect-ProofTh1} (see \mbox{Definition \ref{DefIPseqOrd}}).

\begin{theorem}\label{Th_ZerosSimp}  If the discrete Sobolev inner product  \eqref{GeneralSIP} is sequentially ordered, then  $S_n$ has at least $n-d^*$ changes of sign on $(a,b)$, where $d^*$ is the number of positive coefficients $\lambda_{j,k}$ in  \eqref{GeneralSIP}.
\end{theorem}

Previously, this result was obtained for    more restricted cases  in (\cite{AlfaroLagomasinoRezola}, Th. 2.2) and (\cite{AbelIgnPij20}, Th. 1). In~(\cite{AlfaroLagomasinoRezola}, Th. 2.2), the~authors proved this result for the case $N=1$. In~ (\cite{AbelIgnPij20}, Th. 1), the~notion of a sequentially ordered inner product is more restrictive than here, because~it did not include the case when the Sobolev inner product has more than one derivative order at the same mass point. %Therefore, thus far, our Theorem \ref{Th_ZerosSimp}.

%    Thus far, this result extends the most general results on the location of zeros of discrete Sobolev-type polynomials. In \cite[Th. 2.2]{AlfaroLagomasinoRezola},  the authors prove this result for the case $N=1$ and   in \cite[Th. 1]{AbelIgnPij20} the notion of sequentially ordered inner product is a little more restrictive than here since does not include the case when the Sobolev inner   product  has more than one derivative order at the same mass point.
	
In the second part of this paper,  we focus our attention on the Laguerre--Sobolev-type polynomials (i.e., $d\mu=x^{\alpha}e^{-x}dx$, with~$\alpha>-1$).  In~the case of the inner product,  \eqref{GeneralSIP} takes the form
\begin{align}\label{LaguerreSIP}
	\IpS{f}{g}=\!\!\!\int_0^\infty f(x) g(x)x^{\alpha}e^{-x}dx+\sum_{j=1}^{N} \lambda_{j} f^{(d_j)}(c_j)g^{(d_j)}(c_j),
\end{align}	
where $\lambda_j:=\lambda_{j,d_j}>0$,  $c_j<0$,  for~$j=1,2,\dots,N$, we obtain the outer relative  asymptotic  of the Laguerre--Sobolev-type~polynomials.

\begin{theorem}\label{asymp}
Let $\{L^{\alpha}_n\}_{n\geq 0}$ be the sequence of monic Laguerre polynomials and let $\{S_n\}_{n\geq 0}$ be the monic orthogonal polynomials with respect to the inner product  \eqref{LaguerreSIP}. Then,
\begin{align}\label{compasymt}
\frac{S_n(x)}{\left.L_n^{\alpha}\right.(x)}\rightrightarrows \prod_{j=1}^{N}\left(\frac{\sqrt{-x}-\sqrt{|c_{j}|}}{\sqrt{-x}+\sqrt{|c_{j}|}}\right), \quad K\subset\overline{\CC}\setminus\RRp.
\end{align}
\end{theorem}

Throughout this paper, we use the notation $\dsty f_n \rightrightarrows  f,\;K \subset \mathds{U}$ when the sequence of functions $f_n$ converges to $f$ uniformly on every compact subset $K$ of the region $\mathds{U}$.

 Combining this result with Theorem \ref{Th_ZerosSimp}, we obtain that the Sobolev polynomials $S_n$, orthogonal with respect to a sequentially ordered inner product in the form \eqref{LaguerreSIP}, have at least $n\text{--}N$ zeros in $(0,\infty)$ and, for~sufficiently large $n$, each one of the other $N$ zeros are contained in  a neighborhood of each mass point  $c_j$ ($j=1,\dots,N$). Then, we have located all zeros of $S_n$ and we obtain that for a sufficiently  large $n$, they are simple and real, as in the Krall case (see~\cite{Littlejohn}) or the Krall--Laguerre-type orthogonal polynomial (see~\cite{HMP-PAMS14}). This is summarized in the following corollary.
\begin{corollary}\label{CoroLocZeros}
Let $\mu=\mu_\alpha$ be the classical Laguerre measure ($d\mu_\alpha(x)=x^{\alpha}e^{-x}dx$) and  \eqref{LaguerreSIP} a  sequentially ordered discrete Sobolev inner product. Then, the~following statements~hold:
\begin{enumerate}
\item Every point $c_j$ attracts exactly one zero of $S_n$ for sufficiently large $n$, while the remaining $n\text{--}N$ zeros are contained in $(0,\infty)$.
This means:

For every $r>0$, there exists a natural value $\mathcal{N}$ such that if $n\geq\mathcal{N}$, then the $n$ zeros of $S_n$ $\{\xi_i\}_{i=1}^n$ satisfy
\begin{align*}%\label{atracscondition}
\xi_{j}\in B(c_j,r)\ \text{ for } j=1,\dots, N &&\and{and} && \xi_i\in (0,\infty) \text{ for } i=N+1,N+2,\dots, n.
\end{align*}
\item  The zeros of $S_n$ are real and simple for large-enough values of $n$.
\item The zeros of $\{S_n\}_{n=1}^\infty$ are  at  a finite distance from $(0,\infty)$. This means that there exists a positive constant $M$ such that if $\xi$ is a zero of $S_n$, then
$$\dsty d(\xi,(0,\infty)):=\inf_{x>0}\{|x-\xi|\}<M.$$
\end{enumerate}
\end{corollary}

Section~\ref{Sect-ProofTh1} is devoted to introducing the notion of a sequentially ordered Sobolev inner product   and to prove Theorem \ref{Th_ZerosSimp}. In~Section~\ref{Sec-Laguerre}, we summarize some auxiliary  properties of Laguerre polynomials to be used in the proof of Theorem \ref{asymp}. Some results about the asymptotic behavior of the reproducing kernels are given. The~aim of the last section is to prove Theorem \ref{asymp} and some of its consequences  stated in  Corollary \ref{CoroAsym-all}.

%%%%%%%%%%%%%%%%%%%%%%%%%%%%%%%%%%%%%%%%%%%%%%%%%%%%%%%%%%%%%%%%%%%%%%%%%%%%%%%%%%%%%%%%%
\section{Sequentially Ordered Inner~Product}\label{Sect-ProofTh1}
%%%%%%%%%%%%%%%%%%%%%%%%%%%%%%%%%%%%%%%%%%%%%%%%%%%%%%%%%%%%%%%%%%%%%%%%%%%%%%%%%%%%%%%%%

\begin{definition}[{Sequentially ordered Sobolev inner product}]\label{DefIPseqOrd}
 Consider a discrete Sobolev inner product in the general form  \eqref{GeneralSIP} and assume $d_1\leq d_2\leq \dots\leq d_N$ without loss of generality. We say that a discrete Sobolev inner product is \emph{sequentially ordered} if the conditions
	\begin{equation*}%\label{CondSeqOrdIP}
		\Delta_{k}\cap \inter{\ch{\cup_{i=0}^{k-1}\Delta_i}} =\emptyset, \quad \quad k=1,2,\dots, d_N,
	\end{equation*}		
	hold, where
\begin{equation}\label{DeltaDef}
  \Delta_k=\funD{\ch{\supp{\mu}\cup\{c_j: \lambda_{j,0}>0\}}}{k=0}{\ch{\{c_j: \lambda_{j,k}>0\}}}{1\leq k\leq d_N.}
\end{equation}
\end{definition}

Note that $\Delta_k$ is the convex hull of the support of the measure associated with the $k$-th order derivative in the Sobolev inner product   \eqref{GeneralSIP}. Let us see two~examples.

\begin{example}[Sequentially ordered inner product]\label{ExSeqOrd} 
Let
 \begin{align*}
   \IpS{f}{g}= & \int_{0}^{\infty} f(x) g(x) e^{-x}dx+10f(-1)g(-1)+5f^{\prime}(-3) g^{\prime}(-3) \\
    & +5f^{\prime}(-9) g^{\prime}(-9)+20f^{\prime\prime\prime}(-10) g^{\prime\prime\prime}(-10),
 \end{align*}
 then the corresponding fifth-degree Sobolev orthogonal polynomial has the following exact expression
\begin{align*}
S_5(x)=&x^5+\frac{380961336355365}{16894750106161}x^4+\frac{1836311881214045}{16894750106161}x^3-\frac{7830454972601355}{16894750106161}x^2\\
&-\frac{36972053870326650}{16894750106161}x-\frac{22386262325875230}{16894750106161},
\end{align*}
 whose zeros are approximately $\xi_{1}\approx 4.46 $, $\xi_{2}\approx- 0.74$, $\xi_{3}\approx -2.8$, $\xi_{4}\approx  -11.74+2.51i$ and $\xi_{5}\approx -11.74-2.51i$. Note that four of them are outside of $(0,\infty)$ and two are even complex.
\end{example}

\begin{example}[Non-sequentially ordered inner product]\label{ExNoSeqOrd} Let $$\dsty  \langle f,g \rangle=  \int_{0}^{\infty} f(x) g(x) \,e^{-x}dx+f^{\prime}(-15) g^{\prime}(-15)+f^{\prime\prime}(-9) g^{\prime\prime}(-9), $$then the corresponding fifth-degree Sobolev orthogonal polynomial has the following exact expression
\begin{align*}
  S_5(x)= & x^5+\frac{55079160}{21682477}x^4-\frac{5053767275}{21682477}x^3+\frac{40953207555}{21682477}x^2 \\
    & -\frac{98030649090}{21682477}x+\frac{42523040550}{21682477},
\end{align*}
whose zeros are approximately $\xi_{1}\approx 0.55 $, $\xi_{2}\approx  3.36$, $\xi_{3}\approx 6.66+3.02i$, $\xi_{4}\approx  6.66-3.02i$ and $\xi_{5}\approx -19.77 $. Note that, in~spite of Theorem \ref{Th_ZerosSimp}, $d^*=2$ and three of the zeros of $S_5$ are outside of $(0,\infty)$, with two of them as not even {real.} 
\end{example}

In the sequentially ordered example (Example \ref{ExSeqOrd}), $S_5$ has exactly $1=5-4=n-d^*$ simple zeros on the interior of the convex hull of the support of the Laguerre measure $(0,\infty)$, and thus, the bound of Theorem \ref{Th_ZerosSimp} is sharp. In~addition, this example shows that the remaining $d^*$ zeros might even be complex, although~Corollary \ref{CoroLocZeros} shows that this does not happen when $n$ is sufficiently~large.

On the other hand, in~the non-sequentially ordered example (Example \ref{ExNoSeqOrd}), this condition is not satisfied, since $S_5$ has only $2<3=5-2=n-d^*$ zeros on $(0,\infty)$, showing that the sequential order plays a main role in the localization of the zeros of $S_n$, at~least to obtain this property for every value of $n$.

Throughout the remainder of this section, we will consider inner products of the form  \eqref{GeneralSIP} that are sequentially ordered. The~next lemma is an extension of (\cite{LoPiPe01}, Lemma 2.1) and (\cite{AbelIgnPij20}, Lemma 3.1).

\begin{lemma}\label{Lem-CoroRolle}
	Let $\{I_i\}_{i=0}^m$ be a set of $m+1$ intervals on the real line and let $P$ be a polynomial with real coefficients of degree $\geq m$.  {If} 
 \begin{align}\label{CondSeqOrdSet}
		I_{k}\cap \inter{\ch{\cup_{i=0}^{k-1}I_i}}=\emptyset, \quad \quad k=1,2,\dots, m,
	\end{align}
	then
\begin{align}\nonumber
		\Nzero{P}{J}+\Ncero{P}{I_0\setminus J}+\sum_{i=1}^m \Ncero{P^{(i)}}{I_i}\leq & \ \Nzero{P^{(m)}}{J}\\ \label{IneqRollGeneral}
		  &+\Ncero{P^{(m)}}{\ch{\cup_{i=0}^{m}I_i}\setminus J}+m,
	\end{align}
	for every closed subinterval $J$ of $\inter{I_0}$ (both empty set and unitary sets are assumed to be intervals). Here, given a real set $A$ and a polynomial $P$, $\Ncero{P}{A}$ denotes the number of values where the polynomial $P$ vanishes on $A$ (i.e., zeros of $P$ on $A$ without counting multiplicities), and $\Nzero{P}{A}$ denotes the total number of zeros (counting multiplicities) of $P$ on $A$.
\end{lemma}
\begin{proof}
	First, we point out the following consequence of Rolle's Theorem. If~$I$ is a real interval and $J$ is a closed subinterval of $\inter{I}$, then
\begin{align}\label{IneqRoll}
		\Nzero{P}{J}+\Ncero{P}{I\setminus J}\leq \Nzero{P'}{J}+ \Ncero{P'}{\inter{I}\setminus J}+1.
	\end{align}	
 It is easy to see that	 \eqref{IneqRollGeneral} holds for $m=0$. We now proceed by induction on $m$. Suppose that we have  $m+1$ intervals  $\{I_i\}_{i=0}^m$ satisfying \eqref{CondSeqOrdSet}; thus, the first $m$ intervals $\{I_i\}_{i=0}^{m-1}$ also satisfy  \eqref{CondSeqOrdSet}, and we obtain  \eqref{IneqRollGeneral} by induction  hypothesis (taking $m-1$ instead of $m$). Then
	\begin{align*}
		&\Nzero{P}{J}+\Ncero{P}{I_0\setminus J}+\sum_{i=1}^m \Ncero{P^{(i)}}{I_i},\\
			&\leq\Nzero{P^{(m-1)}}{J}+\Ncero{P^{(m-1)}}{\ch{\cup_{i=0}^{m-1}I_i}\setminus J}+m-1+\Ncero{P^{(m)}}{I_m},\\
			&\leq \Nzero{P^{(m)}}{J}+\Ncero{P^{(m)}}{\inter{\ch{\cup_{i=0}^{m-1}I_i}}\setminus J}+m+\Ncero{P^{(m)}}{I_m},\\
			&\leq \Nzero{P^{(m)}}{J}+\Ncero{P^{(m)}}{\ch{\cup_{i=0}^{m}I_i}\setminus J}+m,
	\end{align*}
	where in the second inequality, we have used  \eqref{IneqRoll}.
\end{proof}

As an immediate consequence of   Lemma \ref{Lem-CoroRolle}, the~following result is~obtained.

\begin{lemma}
Under the assumptions of  Lemma \ref{Lem-CoroRolle}, we have
\begin{align}\label{IneqRollzerocero}
\Nzero{P}{J}+\Ncero{P}{I_0\setminus J}+\sum_{i=1}^m \Ncero{P^{(i)}}{I_i}\leq \deg{P}
\end{align}
for every $J$ closed subinterval of $\inter{I_0}$. In~particular, for~$J=\emptyset$, we obtain
\begin{align}\label{IneqRollcero}
\sum_{i=0}^m \Ncero{P^{(i)}}{I_i}\leq \deg{P}.
\end{align}
\end{lemma}

\begin{lemma}\label{PolMiDeg}
Let $\{(r_i,\nu_i)\}_{i=1}^M\subset \RR\times\ZZp$ be a set of $M$ ordered pairs. Then, there exists a unique monic polynomial $U_M$ of minimal degree (with $0\leq \deg{U_M}\leq M$), such that
\begin{align}\label{cond}
		U_M^{(\nu_i)}(r_i)=0, \quad i=1,2,\dots,M.
	\end{align}
	
	 {Furthermore}, if~the intervals $I_k=\ch{\{r_i: \nu_i=k\}}$, $k=0,1,2,\dots,\nu_M$, satisfy  \eqref{CondSeqOrdSet}, then
 $U_M$ has degree  $\mathfrak{u}_{M}=\min \mathfrak{I}_{M}-1$, where $$\mathfrak{I}_{M}=\{i: 1\leq i\leq M \text{ and } \nu_i\geq i\}\cup \{M+1\}.$$
\end{lemma}

\begin{proof}
	The existence of a nonidentical zero polynomial with degree $\leq M$ satisfying  \eqref{cond} reduces to solving a homogeneous linear system with $M$ equations and $M+1$ unknowns (its coefficients). Thus, a~non-trivial solution always exists. In~addition, if~we suppose that there exist two different minimal monic polynomials $U_M$ and $\widetilde{U}_M$, then  the polynomial $\widehat{U}_M=U_M-\widetilde{U}_M$ is 	not identically zero, it satisfies  \eqref{cond}, and~ $\deg{\widehat{U}_M}<\deg{U_M}$. Thus, if~we divide $\widehat{U}_M$ by its leading coefficient, we reach a~contradiction.

	The rest of the proof runs by induction on the number of points $M$.  For~$M=1$, the~result follows taking
		$$U_1(x)=\funD{x-r_1}{\nu_1=0}{1}{\nu_1\geq 1.}$$

	Suppose that,  for~each sequentially ordered sequence  of $M$ ordered pairs,  the~corresponding minimal  polynomial $U_M$ has degree $\mathfrak{u}_{M}$.

	Let $\{(r_i,\nu_i)\}_{i=1}^{M}$ be a set of $M$ ordered pairs satisfying  \eqref{CondSeqOrdSet}. Obviously,  $\{(r_i,\nu_i)\}_{i=1}^{M-1}$ also satisfies  \eqref{CondSeqOrdSet}  and $U_M$ satisfies \eqref{cond} for $i=1,2,\dots,M-1$; thus, we obtain  $\deg{U_{M-1}}=\mathfrak{u}_{M-1}$ and $\deg{U_{M}}\geq \deg{U_{M-1}}$.  Now, we divide the proof into two~cases:
\begin{enumerate}
	\item If $\mathfrak{u}_{M}=M$, then for all $1\leq i\leq M$ we have $ \nu_i< i$, which yields
		\begin{equation*}
			\deg{U_{M}}\geq \deg{U_{M-1}}=\mathfrak{u}_{M-1}=M-1 \geq \nu_{M}.
		\end{equation*}
		Since $\{(r_i,\nu_i)\}_{i=1}^{M}$ satisfies \eqref{CondSeqOrdSet}, from~\eqref{IneqRollcero} we obtain
			$$ M\leq \sum_{i=0}^{\nu_{M}} \Ncero{U_{M}^{(i)}}{I_i} \leq  \deg{U_{M}},$$
		which implies that  $\deg{U_{M}}=M=\mathfrak{u}_{M}$.
	\item  If $\mathfrak{u}_{M}\leq M-1$, then  there exists a minimal $j$ ($1\leq j\leq M$), such that $\nu_j\geq j$,  and~$ \nu_i< i$ for all 				$1\leq i\leq j-1$. Therefore,  $\mathfrak{u}_{M}=j-1=\mathfrak{u}_{M-1}$. From~the induction hypothesis, we obtain
			$$\deg{U_{M-1}}=\mathfrak{u}_{M-1}=j-1\leq \nu_j-1\leq \nu_{M}-1,$$
		which gives  $U^{(\nu_{M})}_{M-1}\equiv 0$. Hence,  $U_{M}\equiv U_{M-1}$ and, consequently, we obtain
			$$\deg{U_{M}}=\deg{U_{M-1}}=\mathfrak{u}_{M-1}=\mathfrak{u}_{M}.$$
\end{enumerate}
\end{proof}

Note that, in~Lemma \ref{PolMiDeg},  condition \eqref{CondSeqOrdSet} is necessary for asserting that the polynomial $U_M$ has degree $\mathfrak{u}_{M}$. If~we consider $\{(-1,0),(1,0),(0,1)\}$,  whose corresponding convex hulls $I_0=[-1,1]$ and $I_1=\{0\}$ do not satisfy \eqref{CondSeqOrdSet}, we obtain $U_3(x)=x^2-1$ and $\mathfrak{u}_{3}=3\neq \deg{U_3}$.

Now we are able to prove the zero localization theorem for sequentially ordered discrete Sobolev inner~products.

\begin{proof}[Proof of Theorem \ref{Th_ZerosSimp}]
	Let  $\xi_1<\xi_2<\cdots <\xi_{\eta}$ be the points on $(a,b)=\inter{\ch{\supp{\mu}}}$ where $S_n$ changes sign and suppose that $\eta<n-d^*$. Consider the set of ordered pairs
	$$\{(r_i,\nu_i)\}_{i=1}^{d^*+\eta}=\{(\xi_i,0)\}_{i=1}^{\eta}\cup\{(c_j,k): \eta_{j,k}>0, \ j=1,2,\dots,N, k=1,\dots,d_j\}.$$
	
	 {Since}  $\IpS{\cdot}{\cdot}$ is sequentially ordered, the~intervals $I_k=\Delta_k$ for $k=0,1,\dots,\nu_N$ (see \eqref{DeltaDef})  satisfy \eqref{CondSeqOrdSet} (we can assume without loss of generality that $\nu_1\leq\nu_2\leq \cdots\leq \nu_{d^*+\eta}$). Consequently, from~Lemma \ref{PolMiDeg}, there exists a unique monic polynomial $U_{d^*+\eta}$ of minimal degree, such that
\begin{align}\label{CondNull}
		U_{d^*+\eta}(\xi_i)&=0;\qquad\text{for } i=1,\dots, \eta,\nonumber\\
		U_{d^*+\eta}^{(k)}(c_j)&=0;\qquad\text{for each } (j,k): \eta_{j,k}>0,
	\end{align}
	and  $\dsty 		\deg{U_{d^*+\eta}}=\min \mathfrak{I}_{d^*+\eta}-1\leq d^*+\eta$,	where
\begin{equation}\label{DegQ}
		\mathfrak{I}_{d^*+\eta}=\{i: 1\leq i\leq d^*+\eta \text{ and } \nu_i \geq i\}\cup \{d^*+\eta+1\}.
	\end{equation}
	
	 {Now}, we need to consider the following 	two~cases.
	\begin{enumerate}
		\item If $\deg{U_{d^*+\eta}}=d^*+\eta$, from~\eqref{DegQ},  we obtain $\deg{U_{d^*+\eta}}\geq\nu_{\eta+d^*}+1$. 					Thus, taking the closed interval $J=[\xi_1,\xi_\eta]\subset (a,b)$ in \eqref{IneqRollzerocero},  we obtain
			\begin{align*}
				d^*+\eta\leq&\sum_{k=0}^{\nu_{d^*+\eta}} \Ncero{U_{d^*+\eta}^{(k)}}{I_k} \leq \Nzero{U_{d^*+\eta}}{[\xi_1,\xi_\eta]}+\Ncero{U_{d^*+\eta}}{I_0\setminus [\xi_1,\xi_\eta]}\\
				  &+\sum_{k=1}^{\nu_{d^*+\eta}} \Ncero{U_{d^*+\eta}^{(k)}}{I_k}\leq   \deg{U_{d^*+\eta}}=d^*+\eta.
			\end{align*}		
		\item If $\deg{U_{d^*+\eta}}<d^*+\eta$,  from~\eqref{DegQ},  there exists $1\leq j\leq d^*+\eta$ such that $\deg{U_{d^*+\eta}}=j-1$, $\nu_{j}\geq j$ and $\nu_i\leq i-1$ for $i=1,2,\dots,j-1$. Hence, $$\nu_{j-1}+1\leq j-1=\deg{U_{d^*+\eta}}$$  and, again, from~\eqref{IneqRollzerocero} we have
			\begin{align*}
				j-1\leq&\sum_{k=0}^{\nu_{j-1}} \Ncero{U_{d^*+\eta}^{(k)}}{I_k}\leq \Nzero{U_{d^*+\eta}}{[\xi_1,\xi_\eta]}+\Ncero{U_{d^*+\eta}}{I_0\setminus [\xi_1,\xi_\eta]}\\
				  &+\sum_{k=1}^{\nu_{j-1}} \Ncero{U_{d^*+\eta}^{(k)}}{I_k} \leq \deg{U_{d^*+\eta}}=j-1.
			\end{align*}
		\end{enumerate}
		
	 {In}  both cases, we obtain that $U_{d^*+\eta}$ has no other zeros in $I_0$ than those given by construction, and from $\Ncero{U_{d^*+\eta}}{[\xi_1,\xi_\eta]}=\Nzero{U_{d^*+\eta}}{[\xi_1,\xi_\eta]}$, all the zeros of $S_n$ on $\inter{I}$ are simple. Thus, in~addition to \eqref{CondNull}, we obtain that $S_nU_{d^*+\eta}$ does not change sign on $\inter{I}$. Now, since $\deg{U_{d^*+\eta}}\leq d^*+\eta<n$, we arrive at  the contradiction
	\begin{align*}
 		0&=\langle S_n, U_{d^*+\eta}\rangle=\int S_n(x) U_{d^*+\eta}(x)  d\mu(x) +\sum_{j=1}^{N}\sum_{k=0}^{d_j}\lambda_{j,k} S_n^{(k)}(c_{j}) U_{d^*+\eta}^{(k)}(c_{j})\\
 			&=\int_a^b S_n(x) U_{d^*+\eta}(x)   d\mu(x)\neq0.
	\end{align*}
\end{proof}

\section{Auxiliary~Results}\label{Sec-Laguerre}

The family of Laguerre polynomials is one of the three very well-known classical orthogonal polynomials families (see~\cite{Chi78,Freud71,Szego75}).
It consists of the sequence of polynomials $\{L_n^{(\alpha)}\}$ that are orthogonal with respect to the measure $d\mu=x^{\alpha}e^{-x}dx$,  $x\in(0,\infty)$, for~$\alpha >-1$, and that are normalized by taking $\frac{(-1)^{n}}{n!}$ as the leading coefficient of the $n$-th degree polynomial of the sequence.   Laguerre polynomials play a key role in applied mathematics and physics, where they are involved in the solutions of the wave equation of the hydrogen atom (c.f.~\cite{SchRat02}).

Some of the structural properties of this family are listed in the following proposition in order to be used~later.

\begin{proposition}%\label{ProLagu}
 Let $\{L_{n}^{(\alpha )}\}_{n\geq 0}$ (note the brackets in  parameter $\alpha $) be the sequence of Laguerre polynomials and let $\{L_{n}^{\alpha}\}_{n\geq 0}$ be the monic sequence of Laguerre polynomials. Then, the~following statements~hold.
\begin{enumerate}
\item For every $n\in \mathbb{N},$%
\begin{align}
L_{n}^{(\alpha )}\left( x\right) =\frac{\left( -1\right) ^{n}}{n!}%
L_{n}^{\alpha }\left( x\right) .  \label{Lead-Coeff}
\end{align}
\item Three-term recurrence relation. For~every $n\geq 1$,%
\begin{align*}%\label{3TRRLagMonic}
 xL^{\alpha}_{n}(x)&=L^{\alpha}_{n+1}(x)+(2n+\alpha+1)L^{\alpha}_n(x)+n(n+\alpha)L^{\alpha}_{n-1}(x) \\
xL_{n}^{(\alpha )}(x)&=-(n+1)L_{n+1}^{(\alpha )}(x)+(2n+\alpha +1)L_{n}^{(\alpha )}(x)-(n+\alpha)L_{n-1}^{(\alpha )}(x)%\label{RRLagu}
\end{align*}

with $L_{-1}^{(\alpha )}\equiv L_{-1}^{\alpha }=0$, and~$L_{0}^{(\alpha )}\equiv L_{0}^{\alpha }\equiv 1$.

\item Structure relation. For~every $n\in \mathbb{N},$%
\begin{equation*}
L_{n}^{(\alpha )}(x)=L_{n}^{(\alpha +1)}\left( x\right) -L_{n-1}^{(\alpha
+1)}\left( x\right) . 
\end{equation*}

\item For every $n\in \mathbb{N},$%
\begin{equation}\label{NormLag}
||L_{n}^{(\alpha )}||_{\mu }^{2}=\Gamma (\alpha +1)\binom{n+\alpha }{n}=\frac{\Gamma(\alpha+n+1)}{n!}.
\end{equation}%
In addition, we have 
\begin{equation*} 
||L_{n}^{\alpha }||_{\mu }^{2}=n!\Gamma(n+\alpha+1)
\end{equation*}

\item Hahn condition. For~every $n\in \mathbb{N},$%
\begin{equation}
\lbrack L_{n}^{(\alpha )}]^{\prime }(x)=-L_{n-1}^{(\alpha +1)}(x).
\label{Derivada de Ln}
\end{equation}

\item Outer strong asymptotics (Perron's asymptotics formula on $\mathbb{C}%
\setminus \mathbb{R}_{+}$). Let $\alpha \in \mathbb{R}$. Then%
\begin{equation}
L_{n}^{(\alpha )}\left( x\right) =\frac{e^{x/2}n^{\alpha /2-1/4}e^{2\left( -nx\right)
^{1/2}}}{2\pi ^{1/2}\left(
-x\right) ^{\alpha /2+1/4}}\left\{ \sum\limits_{k=0}^{p-1}C_{k}(x)n^{-k/2}+\mathcal{O}%
(n^{-p/2})\right\} .  \label{PerronsForm}
\end{equation}

Here, $\{C_{k}(x)\}_{k=0}^{p-1}$ are certain analytic functions of $x$ independent of $n$, with~$C_0\equiv 1$. This relation holds for $x$\ in the
complex plane with a cut along the positive part of the real axis. The~bound for the
remainder holds uniformly in every closed domain with no points in common
with $x\geq 0$ (see~\cite{Szego75}, Theorem 8.22.3).

\end{enumerate}
\end{proposition}

Now, we summarize some auxiliary  lemmas to be used in the proof of Theorem \ref{asymp} (see (\cite{DHM12}, Lem. 1)  and (\cite{MZFH12}, Prop. 6)).

\begin{lemma} \label{Lem-LagRatioAsymp} For $z\in\CC\setminus[0,\infty)$, $\alpha, \beta \in \RR$ and $j,k\geq-n$ we have
\begin{equation}\label{Lem-LagRatioAsymp-0}
\frac{L^{(\alpha+\beta)}_{n+j}(z)}{L^{(\alpha)}_{n+k}(z)}=\begin{cases} 1+\frac{(j-k)\sqrt{-z}}{\sqrt{n}}+\left(\frac{\alpha}{2}-\frac{1}{4}-z\frac{(j-k)}{2}\right)\frac{(j-k)}{n}+\mathcal{O}_z\left(n^{-\frac{3}{2}}\right) \quad &\text{ if } \beta =0\\
\left(\frac{\sqrt{n}}{\sqrt{-z}} \right)^{\beta} \left( 1+\mathcal{O}_z\left(n^{-1/2}\right)\right)\quad &\text{ if } \beta \neq 0
.\end{cases}
\end{equation}
where $\mathcal{O}_z(n^{-j})$ denotes some analytic function sequence $\{g_n(z)\}_{n=1}^{\infty}$ such that $\{n^{j}g_{n}\}$ is uniformly bounded on every compact subset of $\CC\setminus[0,\infty)$.
\end{lemma}

To study the outer relative asymptotic between the standard Laguerre polynomials and the Laguerre--Sobolev orthogonal polynomials (see Formula \eqref{compasymt}), we need to compute the behavior of the Laguerre kernel polynomials and their derivatives when $n$ approaches infinity. To~this end, we prove  the following auxiliary result, which is an extension of {(\cite{Ahl79}, Ch. 5, Th. 16)}.

\begin{lemma}\label{OgrandeDer}
Let $G$ and $G'$ be two open subsets of the complex plane and $f_n:G\times G'\longrightarrow \CC$ be  a sequence of functions that are analytic with respect to each variable separately. If~$\{f_n\}_{n=1}^{\infty}$ is a uniformly bounded sequence on each set in the form $K\times K'$, where $K\subset G$ and $K' \subset G'$ are compact sets, then any of its partial derivative sequences are also uniformly bounded on each set in the form $K\times K'$.
\end{lemma}

\begin{proof}
Note that it is sufficient  to  prove this for the first derivative order with respect to any of the variables and then proceed by induction.  Let $K\subset G$ and $K' \subset G'$ be  two compact sets. Denote $G^c=\CC \setminus G$, $\dsty d(K,G^c)= \inf_{{z\in K}, {w \in G^c}}|z-w|$,    $r=d(K,G^c)/2>0$ and $B(z,r)=\{\zeta \in \CC: |z-\zeta|<r\}$.  Take $K^*$ as the closure of $\bigcup_{z\in K}B(z,r)$; thus, $K^*$ is a compact subset of $G$. Thus, there exists a positive constant $M>0$ such that $|f_n(z,w)|\leq M$ for all $z\in K^*$, $w\in K'$ and $n\in\NN$. Hence, for all $z\in K$, $w\in K'$ and  $n\in\NN$, we obtain
\begin{align*}
\left|\frac{\partial f_n}{\partial z}(z,w)\right|&=\left|\frac{1}{2\pi i}\int_{c(z,r)}\frac{f_n(\xi, w)}{(\xi-z)^2}d\xi\right|\leq \frac{V(c(z,r))}{2\pi}\max_{\xi\in c(z,r)}\left\{\frac{\left|f_n(\xi,w)\right|}{|\xi-z|^2}\right\}\\
&=\frac{2\pi r}{2\pi r^2}\max_{\xi\in c(z,r)}\left\{\left|f_n(\xi,w)\right|\right\}\leq \frac{M}{r},
\end{align*}
where $c(z,r)$ denotes the circle with center at $z$, radius $r$  and length $V(c(z,r))$. \end{proof}

From the Fourier expansion of $S_{n}$ in terms of the basis $\left\{ L^{\alpha}_n\right\} _{n\geqslant 0}$ we obtain
\begin{align}\nonumber
S_{n}(x)&=\sum_{i=0}^{n} \Ip{S_n}{L^{\alpha}_i}_{\mu } \frac{L^{\alpha}_{i}(x)}{\left\Vert L^{\alpha}_{i}\right\Vert _{\mu }^{2}}
=L^{\alpha}_n(x)+\sum_{i=0}^{n-1}\Ip{S_n}{L^{\alpha}_{i}}_{\mu }\frac{L^{\alpha}_{i}(x)}{\left\Vert L^{\alpha}_{i}\right\Vert _{\mu }^{2}}\\ \nonumber
&=L^{\alpha}_n(x)+\sum_{i=0}^{n-1}\left(\IpS{S_n}{L^{\alpha}_{i}}-\sum_{(j,k)\in I_+}\lambda_{j,k} S_n^{(k)}(c_j)\left(L^{\alpha}_{i}\right)^{(k)}(c_j)\right)\frac{L^{\alpha}_{i}(x)}{\left\Vert L^{\alpha}_{i}\right\Vert _{\mu }^{2}}
\\ \nonumber
&=L^{\alpha}_n(x)-\sum_{(j,k)\in I_+}\lambda_{j,k} S_n^{(k)}(c_j)\sum_{i=0}^{n-1}\frac{L^{\alpha}_{i}(x)\left(L^{\alpha}_{i}\right)^{(k)}(c_j)}{\left\Vert L^{\alpha}_{i}\right\Vert _{\mu }^{2}}
\\ \label{FConexFINAL}
&=L^{\alpha}_{n} (x)-\sum_{(j,k)\in I_+}\lambda_{j,k} S_{n}^{(k)}(c_{j})K_{n-1}^{(0,k)}(x,c_j),
\end{align}
where we use the notation $\dsty K_{n}^{(j,k)}\left( x,y\right) =\frac{\partial ^{j+k}K_{n}\left( x,y\right)
}{\partial ^{j}x\partial ^{k}y}$  to denote the partial derivatives of the kernel polynomials defined in \eqref{Kernel-n}.  Differentiating  Equation  \eqref{FConexFINAL} $\ell$-times and evaluating then at $x=c_{i}$ for each ordered pair $(i,\ell)\in I_+$, we obtain the following system of $d^*$ linear equations and $d^*$ unknowns $S^{(k)}_n(c_{j})$.
\begin{align}%\nonumber
\left(L^{\alpha}_{n}\right)^{(\ell)}(c_{i})=  \left(1+\lambda_{i,\ell}K_{n-1}^{(\ell,\ell)}(c_i,c_i)\right)S^{(\ell)}_n(c_{i})   + \sum_{\substack{(j,k)\in I_+\\ (j,k)\neq (i,\ell)}}^{d_j}\!\!\!\!\!\lambda_{j,k} K_{n-1}^{(\ell,k)}(c_{i},c_j)S^{(k)}_n(c_{j}).\label{eqsystem1}%, \quad (i,\ell)\in I_+
\end{align}

\begin{lemma}%\label{Lem-KernelAsymp}
The Laguerre kernel polynomials and their derivatives satisfy the following behavior when $n$ approaches infinity for $x,y\in \CC\setminus [0,\infty)$
\begin{align*}
K_{n-1}^{(i,j)}\left( x,y\right) =\frac{\partial^{i+j}K_{n-1}}{\partial^ix\partial^jy}(x,y)&=\frac{ L^{(\alpha+i)}_{n}(x)L^{(\alpha+j)}_{n}(y)}{n^{\alpha-\frac{1}{2}}(\sqrt{-x}+\sqrt{-y})}\left((-1)^{i+j}+\mathcal{O}_{x,y}(n^{-1/2})\right),\quad i,j\geq 0,
\end{align*}
where $\mathcal{O}_{x,y}(n^{-k})$ denotes some sequence of functions $\{g_n(x,y)\}_{n=1}^{\infty}$ that are holomorphic with respect to each variable and whose sequence  $\{n^{k}g_{n}\}$ is uniformly bounded on every set $K\times K'$,  such that $K$ and $K'$ are compact subsets of $\CC\setminus\RRp$.
\end{lemma}
\begin{proof}
 The proof is by induction on $k=i+j$. First, suppose $k=0$ (i.e., $i=j=0$) and split the proof into two cases according to whether $x=y$ or not. If~$x=y$, from~ \eqref{CD-Ident}, \eqref{Lead-Coeff}, \eqref{Derivada de Ln} and  \eqref{Lem-LagRatioAsymp-0},  we obtain
\allowdisplaybreaks
\begin{align*}\nonumber
\frac{\|L^{(\alpha)}_{n-1}\|_\mu^2}{n} K_{n-1}(x,x)=&L^{(\alpha)}_{n}(x)(L^{(\alpha)}_{n-1})'(x)-(L^{(\alpha)}_{n})'(x)L^{(\alpha)}_{n-1}(x) \\ = &
L^{(\alpha+1)}_{n-1}(x)L^{(\alpha)}_{n-1}(x)-L^{(\alpha+1)}_{n-2}(x)L^{(\alpha)}_{n}(x)\\
=&L^{(\alpha+1)}_{n-2}(x)L^{(\alpha)}_{n-1}(x)\left(\frac{L^{(\alpha+1)}_{n-1}(x)}{L^{(\alpha+1)}_{n-2}(x)}-\frac{L^{(\alpha)}_{n}(x)}{L^{(\alpha)}_{n-1}(x)}\right)\\
=&L^{(\alpha+1)}_{n-2}(x)L^{(\alpha)}_{n-1}(x)\left[1+\frac{\sqrt{-x}}{\sqrt{n}}+\left[\frac{\alpha+1}{2}-\frac{1}{4}-\frac{x}{2}\right]\frac{1}{n}
+\mathcal{O}_x(n^{-3/2})\right.\\
&-\left.\left(1+\frac{\sqrt{-x}}{\sqrt{n}}+\left[\frac{\alpha}{2}-\frac{1}{4}-\frac{x}{2}\right]\frac{1}{n}+\mathcal{O}_x(n^{-3/2})\right)\right]\\
=&L^{(\alpha+1)}_{n-2}(x)L^{(\alpha)}_{n-1}(x)\left(\frac{1}{2n}+\mathcal{O}_x(n^{-3/2})\right)\\
=&\frac{L^{(\alpha)}_{n}(x)L^{(\alpha)}_{n}(x)}{2n}\left(\frac{\sqrt{n}}{\sqrt{-x}} \right)\left(1+\mathcal{O}_x(n^{-1/2})\right)  \\
=&\frac{L^{(\alpha)}_{n}(x)L^{(\alpha)}_{n}(x)}{2\sqrt{n}\sqrt{-x}}\left(1+\mathcal{O}_{x}(n^{-1/2})\right).
\end{align*}

On the other hand, if~$x\neq y$, from~ \eqref{CD-Ident} and  \eqref{Lem-LagRatioAsymp-0}  we obtain
\begin{align*}
\frac{\|L^{(\alpha)}_{n-1}\|_\mu^2}{n}K_{n-1}(x,y)& =\frac{L^{(\alpha)}_{n-1}(x)L^{(\alpha)}_{n}(y)-L^{(\alpha)}_{n}(x)L^{(\alpha)}_{n-1}(y)}{x-y}
\\  & =\frac{L^{(\alpha)}_{n-1}(x)L^{(\alpha)}_{n-1}(y)}{x-y}\left(\frac{L^{(\alpha)}_{n}(y)}{L^{(\alpha)}_{n-1}(y)}-\frac{L^{(\alpha)}_{n}(x)}{L_{n-1}^{(\alpha)}(x)}\right)\\
&=\frac{L^{(\alpha)}_{n-1}(x)L^{(\alpha)}_{n-1}(y)}{x-y}\left(\frac{\sqrt{-y}-\sqrt{-x}}{\sqrt{n}}+\mathcal{O}_{x,y}(n^{-1})\right)\\
&=\frac{L^{(\alpha)}_{n-1}(x)L^{(\alpha)}_{n-1}(y)}{\sqrt{-x}+\sqrt{-y}}\left(\frac{1}{\sqrt{n}}+\mathcal{O}_{x,y}(n^{-1})\right)\\
&=\frac{L^{(\alpha)}_{n}(x)L^{(\alpha)}_{n}(y)}{\sqrt{n}(\sqrt{-x}+\sqrt{-y})}\left(1+\mathcal{O}_{x,y}(n^{-1/2})\right).
\end{align*}

 {From}  \eqref{NormLag} and (\cite{Rus05}, Appendix, (1.14))
$$\|L^{(\alpha)}_{n-1}\|_\mu^2=\frac{\Gamma(n+\alpha)}{\Gamma(n)}=n^{\alpha}(1+\mathcal{O}(n^{-1})),$$ which proves the case $k=0$. Now, we assume that the theorem is true for $i+j=k$ and we will prove it for $i+j=k+1$. By~the symmetry of the formula, the~proof is analogous when any of the variables increase its derivative order; thus, we only will prove it when the variable $y$ does.

\begin{align*}
\frac{\partial^{k+1}K_{n-1}}{\partial x^i\partial^{j+1}y}(x,y)=&\frac{\partial }{\partial y}\left(\frac{L^{(\alpha+i)}_{n}(x)L^{(\alpha+j)}_{n}(y)}{n^{\alpha-\frac{1}{2}}(\sqrt{-x}+\sqrt{-y})}\left((-1)^{k}+\mathcal{O}_{x,y}(n^{-1/2})\right)\right)\\
=&\frac{L^{(\alpha+i)}_{n}(x)}{n^{\alpha-\frac{1}{2}}}\left[\frac{\partial }{\partial y}\left(\frac{L^{(\alpha+j)}_{n}(y)}{\sqrt{-x}+\sqrt{-y}}\right)\left((-1)^{k}+\mathcal{O}_{x,y}(n^{-1/2})\right)\right.\\
&\left.+\frac{L^{(\alpha+j)}_{n}(y)}{\sqrt{-x}+\sqrt{-y}}\frac{\partial }{\partial y}\left((-1)^{k}+\mathcal{O}_{x,y}(n^{-1/2})\right)\right]\\
=&\frac{L^{(\alpha+i)}_{n}(x)}{n^{\alpha-\frac{1}{2}}}\left[\frac{-(\sqrt{-x}+\sqrt{-y})L^{(\alpha+j+1)}_{n-1}(y)+\frac{1}{2} L^{(\alpha+j)}_{n}(y)(-y)^{-1/2}}{(\sqrt{-x}+\sqrt{-y})^2} \right.\\
&\left.\cdot\left((-1)^{k}+\mathcal{O}_{x,y}(n^{-1/2})\right)+\frac{L^{(\alpha+j)}_{n}(y)}{\sqrt{-x}+\sqrt{-y}}\mathcal{O}_{x,y}(n^{-1/2})\right]\\
=&\frac{L^{(\alpha+i)}_{n}(x)L^{(\alpha+j+1)}_{n-1}(y)}{n^{\alpha-\frac{1}{2}}(\sqrt{-x}+\sqrt{-y})}\left[\left(-1+\frac{\frac{\sqrt{-y}}{\sqrt{n}}+\mathcal{O}_{x,y}(n^{-1})}{2\sqrt{-y}(\sqrt{-x}+\sqrt{-y})}\right)\right.\\
&\left.\cdot\left((-1)^{k}+\mathcal{O}_{x,y}(n^{-1/2})\right)+\left(\frac{\sqrt{-y}}{\sqrt{n}}+\mathcal{O}_{x,y}(n^{-1})\right)\mathcal{O}_{x,y}(n^{-1})\right]\\
=&\frac{L^{(\alpha+i)}_{n}(x)L^{(\alpha+j+1)}_{n-1}(y)}{n^{\alpha-\frac{1}{2}}(\sqrt{-x}+\sqrt{-y})}\left[\left(-1+\mathcal{O}_{x,y}(n^{-1/2})\right)\left((-1)^{k}+\mathcal{O}_{x,y}(n^{-1/2})\right)\right.\\
&\left.+\mathcal{O}_{x,y}(n^{-3/2})\right]\\
=&\frac{L^{(\alpha+i)}_{n}(x)L^{(\alpha+j+1)}_{n}(y)}{n^{\alpha-\frac{1}{2}}(\sqrt{-x}+\sqrt{-y})}\left[(-1)^{k+1}+\mathcal{O}_{x,y}(n^{-1/2})\right],
\end{align*}
where in the third equality we use Lemma \ref{OgrandeDer} to guarantee that $\frac{\partial}{\partial y}\mathcal{O}_{x,y}(n^{-1})=\mathcal{O}_{x,y}(n^{-1})$, and in the fourth equality, we use \eqref{Lem-LagRatioAsymp-0}.
\end{proof}

\section{Proof of Theorem \ref{asymp} and~Consequences}\label{SecOuterCompAsym}
\begin{proof}[Proof of Theorem \ref{asymp}]\

Without loss of generality, we will consider the polynomials $L_n^{(\alpha)}=(-1)^n/n!\, L_n^{\alpha}$ and $\widehat{S}_n=(-1)^n/n! \, S_n$, instead of the monic polynomials $L_n^{\alpha}$ and $S_n$.

Multiplying both sides of \eqref{FConexFINAL} by $(-1)^n/n!$, we obtain
\begin{align}\label{eqkern}
\widehat{S}_{n}(x)=L^{(\alpha)}_{n}(x)-\sum_{j=1}^{N}\lambda_{j}\widehat{S}_{n}^{(d_{j})}(c_{j})K_{n-1}^{(0,d_{j})}(x,c_{j}),
\end{align}

 {Dividing}  by $L^{(\alpha)}_n(x)$ on both sides of  \eqref{eqkern}, we obtain
\begin{align}\label{precompasymt}
\frac{\widehat{S}_n(x)}{L_n^{(\alpha)}(x)}=1-\sum_{j=1}^{N} \lambda_{j} \widehat{S}_n^{(d_j)}(c_{j})\frac{K_{n-1}^{(0,d_{j})}(x,c_{j})}{L_n^{(\alpha)}(x)}
.\end{align}

 {Recall}  that we are considering the Laguerre--Sobolev polynomials $\{\widehat{S}_n\}$ that are orthogonal with respect to  \eqref{LaguerreSIP}. In~this case, the~consistent linear system  \eqref{eqsystem1} becomes
\begin{align}\label{eqsystem2}
\begin{split}
\left(L^{(\alpha)}_{n}\right)^{(d_k)}\!\!(c_{k})\!&=\!\left(\!1\!+\!\lambda_{k}K_{n-1}^{(d_k,d_k)}(c_{k},c_{k})\right)\!\widehat{S}^{(d_k)}_n(c_{k})+\!\sum_{\substack{j=1\\ j\neq k}}^{N}\lambda_{j}K_{n-1}^{(d_k,d_j)}(c_{k},c_{j})\widehat{S}^{(d_j)}_n(c_{j}),
\end{split}
\end{align}
for $k=1,2,\dots,N$. Let us define
\begin{align*}
P^{\alpha}_{n,j}(x):= - \lambda_{j} \widehat{S}_n^{(d_j)}(c_{j})\frac{K_{n-1}^{(0,d_{j})}(x,c_{j})}{L_n^{(\alpha)}(x)}\quad
\text{and} \quad
P^{\alpha}_{j}(x):= \; \lim_{n\to\infty}P^{\alpha}_{n,j}(x).
\end{align*}

 {From}  \eqref{precompasymt}, in~order to prove the existence of the limit  \eqref{compasymt}, we  need to figure out the values of $P^{\alpha}_{j}(x)$. Note that
\begin{align*}
\widehat{S}^{(d_j)}_n(c_{j})=-\frac{L_n^{(\alpha)}(x)P^{(\alpha)}_{n,j}(x)}{\lambda_{j}K_{n-1}^{(0,d_{j})}(x,c_{j})}.
\end{align*}

 {If} we replace these expressions in  \eqref{eqsystem2}, then we obtain the following linear system in the unknowns $P_{n,j}(x)$
\begin{align}\label{systn}
\left(
\begin{array}{cccc}
a_{1,1}(n,x) & a_{1,2}(n,x) &\cdots & a_{1,N}(n,x)\\
a_{2,1}(n,x) & a_{2,2}(n,x) &\cdots & a_{1,N}(n,x)\\
\vdots & \vdots &\ddots & \vdots\\
a_{N,1}(n,x) & a_{N,2}(n,x) &\cdots & a_{N,N}(n,x)
\end{array}
\right)\left(
\begin{array}{c}
P^{\alpha}_{n,1}(x) \\
P^{\alpha}_{n,2}(x)\\
\vdots\\
P^{\alpha}_{n,N}(x)\\
\end{array}
\right)=\left(
\begin{array}{c}
-1\\
-1\\
\vdots\\
-1
\end{array}
\right),
\end{align}
where
\begin{align*}
a_{k,j}(n,x)&=\begin{cases}\frac{L_n^{(\alpha)}(x)K_{n-1}^{(d_k,d_j)}(c_{k},c_{j})}{\left(L^{(\alpha)}_{n}\right)^{(d_k)}\!\!(c_{k})K_{n-1}^{(0,d_j)}\!(x,c_{j})},\quad & j \neq k,\\
\frac{L_n^{(\alpha)}(x)\left(\frac{1}{\lambda_{k}}+K_{n-1}^{(d_k,d_k)}\!(c_{k},c_{k})\right)}{\left(L^{(\alpha)}_{n}\right)^{(d_k)}\!\!(c_{k})K_{n-1}^{(0,d_k)}\!(x,c_{k})},\quad & j=k
.\end{cases}
\end{align*}

 {Now}, we will find the behavior of the coefficients $a_{k,j}(n,x)$ when $n$ approaches infinity. If~ $k=j$, we have
\begin{align*}
a_{k,k}(n,x)&=\frac{L_n^{(\alpha)}(x)\left(\frac{1}{\lambda_{k}}+K_{n-1}^{(d_k,d_k)}\!(c_k,c_k)\right)}{\left(L^{(\alpha)}_{n}\right)^{(d_k)}\!\!(c_{k})K_{n-1}^{(0,d_k)}\!(x,c_{k})}\\
&=\frac{L_n^{(\alpha)}(x)\left(\frac{1}{\lambda_{k}}+\frac{L^{(\alpha+d_k)}_{n}(c_{k})L^{(\alpha+d_k)}_{n}(c_{k})}{n^{\alpha-\frac{1}{2}}\sqrt{-c_{k}}+\sqrt{-c_{k}}}\left((-1)^{d_k+d_k}+\mathcal{O}(n^{-1/2})\right)\right)}{(-1)^{d_k}L^{(\alpha+d_k)}_{n-d_k}(c_{k})\frac{L^{(\alpha+d_k)}_{n}(c_{k})L^{(\alpha)}_{n}(x)}{n^{\alpha-\frac{1}{2}}(\sqrt{-x}+\sqrt{-c_{k}})}\left((-1)^{d_k}+\mathcal{O}_x(n^{-1/2})\right)}\\
&=\frac{\sqrt{-x}+\sqrt{-c_{k}}}{2\sqrt{-c_{k}}}\frac{\left(\frac{n^{\alpha-\frac{1}{2}}}{\lambda_{k}L^{(\alpha+d_k)}_{n}(c_{k})}+L^{(\alpha+d_k)}_{n}(c_{k})\left(1+\mathcal{O}(n^{-1/2})\right)\right)}{L^{(\alpha+d_k)}_{n-d_k}(c_{k})\left(1+\mathcal{O}_x(n^{-1/2})\right)}\\
&=\frac{\sqrt{-x}+\sqrt{-c_{k}}}{2\sqrt{-c_{k}}}\frac{\left(\frac{n^{\alpha-\frac{1}{2}}}{\lambda_{k}\left(L^{(\alpha+d_k)}_{n}(c_{k})\right)^2}+1+\mathcal{O}(n^{-1/2})\right)}{1+\mathcal{O}_x(n^{-1/2})}\\
&=\frac{\sqrt{-x}+\sqrt{-c_{k}}}{2\sqrt{-c_{k}}}\frac{1+\mathcal{O}(n^{-1/2})}{1+\mathcal{O}_x(n^{-1/2})},
\end{align*}
where in the last equality we use Perron's Asymptotic Formula  \eqref{PerronsForm} to obtain
\begin{align*}
\frac{n^{\alpha-\frac{1}{2}}}{(L^{(\alpha+d_k)}_n(c_{k}))^2}&=\frac{4\pi n^{\alpha-\frac{1}{2}}}{e^{c_{k}+4\sqrt{-c_{k}}\sqrt{n}}}\frac{(-c_{k})^{\alpha+d_k+\frac{1}{2}}}{n^{\alpha+d_k-\frac{1}{2}}}\mathcal{O}(1)
=\frac{1}{n^{d_k}e^{4\sqrt{-c_{k}}\sqrt{n}}}\mathcal{O}(1),
\end{align*}
which has exponential decay ($c_k<0$). On~the other hand, if~$k\neq j$, we obtain
\begin{align*}
a_{k,j}(n,x)&=\frac{L_n^{(\alpha)}(x)K_{n-1}^{(d_k,d_j)}(c_{k},c_{j})}{\left(L^{(\alpha)}_{n}\right)^{(d_k)}\!\!(c_{k})K_{n-1}^{(0,d_j)}(x,c_{j})} \\ & =\frac{\frac{L^{(\alpha+d_k)}_{n}(c_{k})}{\sqrt{-c_{k}}+\sqrt{-c_{j}}}\left((-1)^{d_k+d_j}+\mathcal{O}(n^{-1/2})\right)}{(-1)^{d_k}\frac{L^{(\alpha+d_k)}_{n-d_k}(c_{k})}{\sqrt{-x}+\sqrt{-c_{j}}}\left((-1)^{d_j}+\mathcal{O}(n^{-1/2})\right)}\\
&=\frac{\sqrt{-x}+\sqrt{-c_{j}}}{\sqrt{-c_{k}}+\sqrt{-c_{j}}}\frac{\left(1+\mathcal{O}(n^{-1
/2})\right)}{\left(1+\mathcal{O}(n^{-1/2})\right)}.
\end{align*}

 {Hence,} 
\begin{align*}
\lim_{n\to\infty}a_{k,j}(n,x)=\begin{cases}\frac{\sqrt{-x}+\sqrt{-c_{j}}}{\sqrt{-c_{k}}+\sqrt{-c_{j}}},\quad& \text{if } j\neq~k\\
\frac{\sqrt{-x}+\sqrt{-c_{k}}}{2\sqrt{-c_{k}}},\quad& \text{if }j=k
\end{cases}=\frac{\sqrt{-x}+\sqrt{|c_{j}|}}{\sqrt{|c_{k}|}+\sqrt{|c_{j}|}}.
\end{align*}

 {Next}, taking limits on both sides of  \eqref{systn} when $n$ approaches $\infty$, we obtain
\begin{align*}
\def\arraystretch{1.9}
\left(
\begin{array}{cccc}
\frac{\sqrt{-x}+\sqrt{|c_{1}|}}{\sqrt{|c_{1}|}+\sqrt{|c_{1}|}}&\frac{\sqrt{-x}+\sqrt{|c_{2}|}}{\sqrt{|c_{1}|}+\sqrt{|c_{2}|}} &\cdots &\frac{\sqrt{-x}+\sqrt{|c_{N}|}}{\sqrt{|c_{1}|}+\sqrt{|c_{N}|}}\\
\frac{\sqrt{-x}+\sqrt{|c_{1}|}}{\sqrt{|c_{2}|}+\sqrt{|c_{1}|}}&\frac{\sqrt{-x}+\sqrt{|c_{2}|}}{\sqrt{|c_{2}|}+\sqrt{|c_{2}|}} &\cdots &\frac{\sqrt{-x}+\sqrt{|c_{N}|}}{\sqrt{|c_{2}|}+\sqrt{|c_{N}|}}\\
 \vdots & \vdots & \ddots & \vdots\\
 \frac{\sqrt{-x}+\sqrt{|c_{1}|}}{\sqrt{|c_{N}|}+\sqrt{|c_1|}}&\frac{\sqrt{-x}+\sqrt{|c_{2}|}}{\sqrt{|c_{N}|}+\sqrt{|c_{2}|}}&\cdots& \frac{\sqrt{-x}+\sqrt{|c_{N}|}}{\sqrt{|c_{N}|}+\sqrt{|c_{N}|}}\\
\end{array}
\right)\left(
\begin{array}{c}
P^{\alpha}_{1}(x) \\
P^{\alpha}_{2}(x) \\
\vdots\\
P^{\alpha}_{N}(x)
\end{array}
\right)=\left(
\begin{array}{c}
-1\\
-1\\
\vdots\\
-1
\end{array}
\right).
\end{align*}

 {Using}  Cauchy determinants, it is not difficult to prove that the $N$ solutions of the above linear system are
\begin{align*}
P^{\alpha}_{j}(x)=\frac{-2\sqrt{|c_{j}|}}{\sqrt{-x}+\sqrt{|c_{j}|}} \prod_{\substack{l=1\\l\neq j}}^{N}\left(\frac{\sqrt{|c_{j}|}+\sqrt{|c_{l}|}}{\sqrt{|c_{j}|}-\sqrt{|c_{l}|}}\right). 
\end{align*}

 {Now}, from~ \eqref{precompasymt}, we obtain
\begin{align*}
\lim_{n\to\infty}\frac{\widehat{S}_n(x)}{L_n^{(\alpha)}(x)}=1+\sum_{j=1}^{N}\frac{2\sqrt{|c_{j}|}}{\sqrt{-x}+\sqrt{|c_{j}|}}  \prod_{\substack{l=1\\l\neq j}}^{N}\left(\frac{\sqrt{|c_{j}|}+\sqrt{|c_{l}|}}{\sqrt{|c_{j}|}-\sqrt{|c_{l}|}}\right).
\end{align*}

 {If}  we consider the change of variable $z=\sqrt{-x}$ and for simplicity we also consider the notation $t_{j}=\sqrt{|c_{j}|}$, then we obtain the following partial fraction decomposition
\begin{align*}
1+\sum_{j=1}^{N}\frac{2t_{j}}{z+t_{j}}  \prod_{\substack{l=1\\l\neq j}}^{N}\left(\frac{t_{j}+t_{l}}{t_{j}-t_{l}}\right).
\end{align*}

 {Thus,}  we only have to prove that this is the partial fraction decomposition of
\begin{align*}
\prod_{j=1}^{N}\left(\frac{z-t_{j}}{z+t_{j}}\right).
\end{align*}

 {Let}  $P_N(z)=\prod_{j=1}^{N}(z-t_{j})$ and $Q_N(z)=\prod_{j=1}^{N}(z+t_{j})$, then
\begin{align*}
\prod_{j=1}^{N}\left(\frac{z-t_{j}}{z+t_{j}}\right)=\frac{P_{N}(z)}{Q_{N}(z)}=1+\frac{P_{N}(z)-Q_{N}(z)}{Q_{N}(z)}=1+\sum_{j=1}^{N}\frac{A_{j}}{z+t_{j}},
\end{align*}
where
\begin{align*}
A_{j}&=\lim_{z\to-t_{j}}(z+t_{j})\frac{P_N(z)-Q_{N}(z)}{Q_{N}(z)}=\frac{P_{N}(-t_{j})-Q_{N}(-t_{j})}{Q'_{N}(-t_{j})}\\
&=\frac{\dsty \prod_{l=1}^{N}(-t_{j}-t_{l})-\prod_{l=1}^{N}(-t_{j}+t_{l})}{ \dsty \prod_{\substack{l=1\\l\neq j}}^{N}(-t_{j}+t_{l})}=\frac{(-1)^N2t_{j}}{(-1)^N}\prod_{\substack{l=1\\l\neq j}}^{N}\left(\frac{t_{j}+t_{l}}{t_{j}-t_{l}}\right)
=2t_{j}\prod_{\substack{l=1\\l\neq j}}^{N}\left(\frac{t_{j}+t_{l}}{t_{j}-t_{l}}\right),
\end{align*}
which completes the proof.
\end{proof}
%%%%%%%%%%%%%%%%%%%%%%%%%%%%%%%%%%%%%%%%%%%%%%%%%%%%%%%%%%%%%%%%%%%%%%%%%%%%%%%%%%%%%%%%%%%%%%%%%%%%%%%
Obviously, the~inner product \eqref{LaguerreSIP} and the monic polynomial $S_n$ depend on the parameter $\alpha>-1$, so that in what follows, we will denote $S^{\alpha}_{n}=S_{n}$. Formula \eqref{compasymt} allows us to obtain other asymptotic formulas for the polynomials $S^{\alpha}_{n}$.  Three of them are included in the following~corollary.

\begin{corollary}\label{CoroAsym-all} Let $\alpha, \beta >-1$, $n\in \ZZp$ and  $k\geq-n$. Under~the hypotheses of Theorem  \ref{asymp}, we obtain
\begin{align}
  (1)\quad  & \label{Asymt-DiferentParam-1}
\frac{S^{\alpha+\beta}_{n+k}(z)}{n^{k+\beta/2} \;L^{\alpha}_{n}(z)}\rightrightarrows \left(-1\right)^k
\left(\sqrt{-z} \right)^{-\beta}\; \prod_{j=1}^{N}\left(\frac{\sqrt{-x}-\sqrt{|c_{j}|}}{\sqrt{-x}+\sqrt{|c_{j}|}}\right),  \qquad K\subset\overline{\CC}\setminus\RRp. \\
  (2) \quad  & \label{Asymt-DiferentParam-2}
\frac{S^{\alpha+\beta}_{n+k}(z)}{n^{k+\beta/2} \;S^{\alpha}_{n}(z)}\rightrightarrows \left(-1\right)^k
\left(\sqrt{-z} \right)^{-\beta},  \qquad K\subset\overline{\CC}\setminus\RRp.\\
  (3) \quad  & \label{Asymt-DiferentParam-3}  \frac{\left(S_{n}^{\alpha}(z)\right)^{(\nu)}}{\left(L_{n}^{\alpha}(z)\right)^{(\nu)}}    \rightrightarrows  \prod_{j=1}^{N}\left(\frac{\sqrt{-x}-\sqrt{|c_{j}|}}{\sqrt{-x}+\sqrt{|c_{j}|}}\right), \quad K\subset\overline{\CC}\setminus\RRp.
\end{align}
\end{corollary}
%%%%%%%%%%%%%%%%%%%%%%%%%%%%%%%%%%%%%%%%%%%%%%%%%%%%%%%%%%%%%%%%%%%%%%%%%%%%%%%%%%%%%%%%%%%%%%%%%%%%%%%
\begin{proof} Formulas \eqref{Asymt-DiferentParam-1} and \eqref{Asymt-DiferentParam-2} are direct consequences of Theorem \ref{asymp} and  Lemma \ref{Lem-LagRatioAsymp}.

The proof  of  \eqref{Asymt-DiferentParam-3} is by induction on $\nu$. Of~course, \eqref{compasymt} is \eqref{Asymt-DiferentParam-3} for $\nu=0$. Assume that \eqref{Asymt-DiferentParam-3} is true for $\nu=\kappa\geq 0$. Note that
\begin{align*}
\frac{\left(S_{n}^{\alpha}(z)\right)^{(\kappa+1)}}{\left(L_{n}^{\alpha}(z)\right)^{(\kappa+1)}} =   &  \frac{\left(L_{n}^{\alpha}(z)\right)^{(\kappa)}}{\left(L_{n}^{\alpha}(z)\right)^{(\kappa+1)}}\;  \left(\frac{\left(S_{n}^{\alpha}(z)\right)^{(\kappa)}}{\left(L_{n}^{\alpha}(z)\right)^{(\kappa)}}\right)^\prime + \frac{\left(S_{n}^{\alpha}(z)\right)^{(\kappa)}}{\left(L_{n}^{\alpha}(z)\right)^{(\kappa)}}
\end{align*}

From \eqref{Lead-Coeff},  \eqref{Derivada de Ln} and Lemma \ref{Lem-LagRatioAsymp}
\begin{align*}
   \frac{\left(L_{n}^{\alpha}(z)\right)^{(\kappa)}}{\left(L_{n}^{\alpha}(z)\right)^{(\kappa+1)}} =& \frac{L_{n-\kappa}^{(\alpha+\kappa)}}{L_{n-\kappa-1}^{(\alpha+\kappa+1)}}\rightrightarrows 0,  \qquad K\subset\overline{\CC}\setminus\RRp.
\end{align*}

Hence, from~  Theorem \ref{asymp},  we obtain \eqref{Asymt-DiferentParam-3} for $\nu=\kappa+1.$
\end{proof}
%%%%%%%%%%%%%%%%%%%%%%%%%%%%%%%%%%%%%%%%%%%%%%%%%%%%%%%%%%%%%%%%%%%%%%%%%%%%%%%%%%%%%%%%%%%%%%%%%%%%%%%%

\end{document}